\documentclass{amsart}

\usepackage[shortlabels]{enumitem}
\usepackage{amsmath}  

\usepackage[pagebackref=false,colorlinks=true, pdfstartview=FitV, linkcolor=blue,citecolor=red, urlcolor=blue]{hyperref}

\usepackage{amsmath,amsthm,amssymb,amsfonts}
\usepackage{verbatim}
\usepackage{enumitem}
\usepackage{mathtools}
\usepackage{listings}
\usepackage{tikz}
\usepackage{mathrsfs}
\usepackage{upgreek}
\usepackage{setspace}

\usepackage[margin=1in]{geometry} 

\usepackage{marginnote}
\usepackage{changepage}

\usepackage{chngcntr}

\usetikzlibrary{decorations.markings}
\usetikzlibrary{arrows}

\usepackage[math]{cellspace}
\theoremstyle{definition} 
\newtheorem{theorem}{Theorem}

\newtheorem{proposition}[theorem]{Proposition}

\newtheorem{remark}[theorem]{Remark}

\makeatletter

\title{A Classification of the Isomorphism Types of Indecomposable and Simple Modules that Refines the Green Theory in Finite Group Modular Representation Theory}
\author{Morton E. Harris}

\address{Department of Mathematics, Statistics, and Computer Science (M/C 249), University of Illinois at Chicago, 851 South Morgan Street, Chicago, IL, 60607-7045 U.S.A.}

\email{meharris@uic.edu}



\begin{document}

\begin{abstract}
In Finite Group Modular Representation Theory, the basic objects are the indecomposable and simple modules.
This paper offers a new classification of these objects that refines the Green Theory Classification of indecomposable and simple modules in which the coefficient ring is an algebraic closure of a field of prime $p$ elements.
Note that then the indecomposable (resp. simple) modules are absolutely indecomposable (resp. absolutely simple).
The sets of isomorphism tyes of these modules is decomposed into disjoint, non-empty subsets such that any two elements in any subset share Green Theory Invariants.
We also prove a new formula for the number of isomorphism types of absolutely simple modules of finite groups in prime characteristic.
\end{abstract}
	
	\maketitle

Let $G$ be a finite group, let $F$ be a field of prime $p$ elements and let $\bar{F}$ be an algebraic closure of $F$.
In this article we show that the set of representatives of the isomorphism types of finitely generated indecomposable (resp. simple) $\bar{F}G$-modules, $\text{ITI}(\bar{F}G)$, (resp. $\text{ITS}(\bar{F}G)$) decomposes into disjoint non-empty subsets indexed by the set of representatives of the isomorphism types of finitely generated indecomposable (resp. simple) $FG$-modules, $\text{ITI}(FG)$, (resp. $\text{ITS}(FG)$).
Moreover the vertices of the vertices of the labelling indecomposable (resp. simple) $FG$-modules are the vertices of each module in the corresponding subset and there are other Green Theory connections.

After receiving the e-mail \cite{5}, the author observed that important hypotheses were carelessly omitted in \cite[Proposition 2.7]{2} so that the results of \cite[Section 3]{2} were not proved.
Subsequently he realized that:
\begin{enumerate}[(1)]
\item Every element of $\bar{F}$ is algebraic over $F$ so that every finitely generated subfield of $\bar{F}$ is finite.
\end{enumerate}

Since the hypotheses of \cite[Section 3]{2} include that $F$ and $\bar{F}$ are as above, (1) replaces the reerence to \cite[Proposition 2.7]{2} to complete the proofs of the statements in \cite[Section 3]{2}.

This article is devoted to clarifying this situation which refines the fundamental Green Theory of Invariants of Indecomposable Modules in Finite Group Modular Representation Theory under the fields hypothesis above.
Note that then the indecomposable (resp. simple) modules are absolutely indecomposable (resp. absolutely simple).

This paper is devoted to proving the results of \cite[Section 3]{2} and to extending the basic results to finitely generated indecomposable modules.

These results refine the Green Theory of indecomposable modules.
Also we obtain a new formula for the finite number of isomorphism types of simple $\bar{F}G$-modules for any finite group $G$.

Our notation and terminology are standard and tend to follow \cite{1}, \cite{3} and \cite{4}.
All rings have identities, are noetherian and all modules over a ring are unitary and finitely generated right modules.

Let $R$ be a ring.
Then mod-$R$ will denote the abelian category of right $R$-modules.
If $U$ and $V$ are (right) $R$-modules, then $U \mid V$ signifies that $U$ is isomorphic to a direct summand of $V$ in mod-$R$.
If $U \mid V$ and $U$ is indecomposable in mod-$R$, then $U$ is said to be a component of $V$.

We shall often, and without mention, identify isomorphic modules in mod-$R$.

Throughout this paper $G$ will be a finite group and $\Phi$ will be a field of prime characteristic $p$ and $F$ will denote the unique minimal subfield of $\Phi$ of order $p$.

Also $K$ will denote a subfield of $\Phi$ and $KG$ will denote the group algebra of $G$ over $K$.
Note that the consequences of the basic Krull--Schmidt Theorem (\cite[I, Theorem 2.3]{3}) hold in the category of right $KG$-modules.
We shall frequently apply this Theorem without reference.
Also we shall assume that the results of \cite[Section 1 and Section 2 through Theorem 2.6]{2} hold.

Let $\text{ITI}(KG)$ (resp. $\text{(ITS}(KG)$) denote a set of representatives of the isomorphism types of indecomposable (resp. simple) modules in mod-$KG$.

As in \cite[Section 2]{2}, let $\Omega(G)$ be a set of representatives of the isomorphism types of indecomposable $KG$-modules for all subfields $K$ of $\Phi$ and let $\Delta$ be the set of finite subfields of $\Phi$ so that $F \in \Delta$.
Let
$$
I(\Omega(G)) = \{(K,V) \mid K \text{ is a subfield of } \Phi \text{ and } V \in \text{ITI}(KG)\}
$$
and let $S(\Omega(G)) = \{(K,V) \in I(\Omega(G)) \mid V \in \text{ITS}((KG)\}$.
Define a relation $\uparrow$ on $I(\Omega(G))$ by:

if $(K,V)$ and $(L,U) \in I(\Omega(G))$, then $(K,V) \uparrow (L,U)$ if $K$ is a subfield of $L$ and $U \mid V \otimes_K L$ in mod-$LG$ (i.e., $U$ is a component of $V \otimes_K L$ in mod-$LG$).

Thus $\uparrow$ is a reflexive and transitive relation on $I(\Omega(G))$.

Set $\text{FI}(\Omega(G)) = \{(K,V) \in I(\Omega(G)) \mid K \in \Delta \}$ and $\text{FS}(\Omega(G)) = \{(K,V) \in \text{FI}(\Omega(G)) \mid V \in \text{ITS}(KG)\}$.

Let $W \in \text{ITI}(FG)$ and set $\mathcal{E}(W) = \{(K,V) \in \text{FI}(\Omega(G)) \mid (F,W) \uparrow (K,V)\}$ and let $(Q,\mathscr{W})$ be a vertex-source pair of $W$ and set $H = N_G(Q)$.

We shall require (cf. \cite[Section 2]{2} and \cite[Section 26]{1}):

\begin{enumerate}

\item[(2)]
If $(K,V)$ and $(L,U) \in \text{FI}(\Omega(G))$ with $K \leq L$, then $(K,V) \uparrow (L,U)$ if and only if $V \mid \text{Res}_K^L(U)$ in mod-$KG$.

\begin{proof}
\cite[Proposition 1.14 (a) and (b)]{2} suffice.
\end{proof}

\item[(3)] $\mathcal{E}(W) \neq \varnothing$

\begin{proof}
Since $(F,W) \in \mathcal{E}(W)$, we are done.
\end{proof}

\item[(4)] If $(K,V) \in \mathcal{E}(W)$, then $\text{Res}_F^K(V) \cong s W$ in mod-$FG$ for some positive integer $s$.

\begin{proof}
\cite[Proposition 1.14]{2} yields a proof.
\end{proof}

\item[(5)] If $(K,V) \in \text{FI}(\Omega(G))$ and $V$ is a component of $W \otimes_F K$ in mod-$KG$, then $\{(K,V^\sigma) \mid \sigma \in \text{Aut}(K) \} \subseteq \mathcal{E}(W)$ and every element of $\mathcal{E}(W)$ with first component $K \in \Delta$ equals $(K,V^\sigma)$ for some $\sigma \in \text{Aut}(K)$.

\begin{proof}
Here \cite[VII, Theorem 1.20]{4} applies.
\end{proof}

\item[(6)] If $(K,V) \uparrow (L,U)$ in $\text{FI}(\Omega(G))$, then $(K,V) \in \mathcal{E}(W)$ if and only if $(L,U) \in \mathcal{E}(W)$.

\begin{proof}
Clearly (2) and (4) yield a proof.
\end{proof}

\item [(7)] \label{(7)}
If $(K,V) \in \mathcal{E}(W)$, then $Q$ is a vertex of $V$, there is a $Q$-source $U$ of $V$ such that $U \mid \mathscr{W} \otimes_F K$ in mod-$KQ$ and $\text{Gr}_G^H(V) \mid \text{Gr}_G^H(W) \otimes_F K$ in mod-$KH$ where $\text{Gr}_G^H$ denotes the Green correspondence from $G$ to $H = N_G(Q)$.

\begin{proof}
	Here \cite[Proposition 1.13]{2} completes the proof.
\end{proof}

\item[(8)]
The following three conditions are equivalent:

\begin{enumerate}[(i)]

\item $W$ is a simple $FG$-module;

\item There is a $(K,V) \in \mathcal{E}(W)$ such that $V$ is a simple $KG$-module; and

\item If $(K,V) \in \mathcal{E}(W)$, then $V$ is a simple $KG$-module.

\end{enumerate}

\begin{proof}
Here \cite[Lemma 1.15 and Theorem 1.16]{4} suffice.
\end{proof}

\end{enumerate}

\begin{proposition}
$\text{Gr}_G^H(W)$ is an indecomposable $FH$-module with vertex-source pair $(Q,\mathscr{W})$ and $\text{Gr}_G^H$ yields a bijection $: \mathcal{E}(W) \to \mathcal{E}(\text{Gr}_G^H(W))$.
\end{proposition}

\begin{proof}
It is easy to see that \cite[Propositition 1.12 (c) and Proposition 1.13 (b) ]{2} imply the Proposition.
\end{proof}

For the remainder of the article, as in \cite[Section 3]{2}, assume that $\Phi = \bar{F}$ is an algebraic closure of $F$.

\begin{theorem}
\label{theorem2}
Let $J$ be a subfield of $\bar{F}$ and let $X$ be an indecomposable $JG$-module with vertex $Q$.
Then there is a $(K,Y) \in \text{FI}(\Omega(G))$ where $K$ is a finite subfield of $J$, $Q$ is a vertex of $Y$ in mod-$KG$ and $Y \otimes_K J \cong X$ in mod-$JG$.
\end{theorem}

\begin{proof}
Let $B$ be a $J$-basis of $X$ and let $A$ denote the subfield of $J$ generated by the coefficients of $g$ on $B$ for all $g \in G$ so that $A$ is a finite subfield of $J$ by (1).

Let $T$ be a right transversal of $Q$ in $G$ with $1 \in T$ so that $G = \bigcup_{t \in T}(Qt)$ is disjoint.

Here $\text{Ind}_Q^G(\text{Res}_Q^G(X) \cong \bigoplus_{t \in T}(\text{Res}_Q^G(X) \otimes t)$ in mod-$JG$ and $B^* = \{b \otimes t \mid b \in B \text{ and } t \in T\}$ is a $J$-basis of $\bigoplus_{t \in T} (\text{Res}_Q^G(X) \otimes t)$.

Clearly we have the canonical $JG$-module surjection $\pi : \bigoplus_{t \in T}(\text{Res}_Q^G(X) \otimes t) \to X$ such that $x \otimes t \mapsto xt$ for all $x \in X$ and $t \in T$.
Also, since $X$ is $Q$-projective, there is a $JG$-module injection $\phi : X \to \bigoplus_{t \in T} (\text{Res}_Q^G(X) \otimes t)$ such that $\pi \circ \phi = \text{Id}_X$.

Let $C$ denote the (finite) subfield of $J$ generated by the coefficients of $\phi(b)$ on $B^*$ for all $b \in B$ and let $K$ denote the subfield of $J$ generated by $A$ and $C$ so that $K$ is a finite subfield of $J$ by (1).

Set $Y = \bigoplus_{b \in B}(bK)$ so that $Y$ is a $KG$-module.
Also set $Z = \bigoplus_{t \in T}(\text{Res}_Q^G(Y) \otimes t)$ so that $Z$ is also a $KG$-module.
Here $Y \otimes_K J \cong \bigoplus_{b\in B} (bJ) \cong X$ in mod-$JG$ and $\Pi$ induces a $KG$-module surjection $\Pi^* : Z = \bigoplus_{t \in T}(\text{Res}_Q^G(Y) \otimes t) \to Y$ since $Y = \bigoplus_{b \in B}(bK)$ is a $KG$-module.

Also $\phi(b) \in \Sigma_{t \in T}(Y \otimes t)$ for all $b \in B$ so that $\phi$ induces a $KG$-module injection $\phi^* : Y \to \bigoplus_{t \in T}(\text{Res}_Q^G(Y) \otimes t)$ in mod-$JK$.
Thus $\Pi^* \circ \phi^* = \text{Id}_Y$.
Also $Y$ is a $Q$-projective $KG$-module and $X \cong Y \otimes_K J$ in mod-$JG$.
Since $X$ has $Q$ as a vertex, \cite[Lemma 1.6 (b)]{2} implies that $(K,Y) \in \text{FI}(\Omega(G))$ and $Q$ is a vertex of $Y$ which completes the proof.
\end{proof}

\begin{proposition}
\label{proposition3}
Let $J$ be a subfield of $\bar{F}$ and let $X$ be a component of $W \otimes_F J$ in mod-$JG$.
Then there is a $(K,V) \in \mathcal{E}(W)$ such that $K$ is a finite subfield of $J$ and $V \otimes_K J \cong X$ in mod-$JG$.
\end{proposition}

\begin{proof}
By Theorem \ref{theorem2} there is a $(K,V) \in \text{FI}(\Omega(G))$ with $K$ a finite subfield of $J$ such that $V \otimes_K J \cong X$ in mod-$JG$.
Here $X \mid W \otimes_F K \otimes_K J$ so that \cite[VII, Theorem 1.21]{4} implies that $V \mid W \otimes_F K$ in $K$-mod.
Thus $(F,W) \uparrow (K,V)$, $(K,V) \in \mathcal{E}(W)$ and we are done.
\end{proof}

From \cite[VII, Lemma 2.2 and 6.7 (b)]{4} we have:

\begin{remark}
An indecomposable (resp. simple) $\bar{F}G$-module is an absolutely indecomposable (resp. simple) module.
\end{remark}

\begin{remark}
Let $W \in \text{ITS}(FG)$ and $(K,V) \in \mathcal{E}(W)$. Then:
\begin{enumerate}[(a)]

\item $V \in \text{ITS}(KG)$;

\item If $Y \in \text{ITI}(\bar{F}G)$ is such that $Y \cong V \otimes_K \bar{F}$ in mod-$\bar{F}G$, then $Y \in \text{ITS}(\bar{F}G)$; and

\item $W$ is a simple $FG$-module.
\end{enumerate}
\end{remark}

\begin{proof}
Clearly (7) implies (a) and \cite[VII, Lemma 1.15]{4} and (8) complete the proof.
\end{proof}

In our final two results we complete the proofs of and extend the results of \cite[Section 3]{2}.
These results refine the Green Theory of indecomposable modules.

\begin{theorem}
\begin{enumerate}[(a)]

\item The map $\Gamma : \text{ITI}(\bar{F}G) \to \text{ITI}(FG)$ such that if $Y \in \text{ITI}(\bar{F}G)$, then $\Gamma(Y) = W \in \text{ITI}(FG)$ if there is a $(K,V) \in \mathcal{E}(W)$ such that $V \otimes_K \bar{F} \cong Y$ in mod-$\bar{F}G$ is well-defined.
In which case, $Q$ is a vertex of $V$ and $Y$ and if $(Q,\mathscr{V})$ is a vertex-source pair of $V$, if $(Q,\mathscr{Y})$ is a vertex-source pair of $Y$, and if $H = N_G(Q)$, then $\mathscr{Y}^h \mid \mathscr{V} \otimes_K \bar{F}$ in mod-$\bar{F}Q$ for some $h \in H$ and $\text{Gr}_G^H(Y) \mid \text{Gr}_G^H(V) \otimes_K \bar{F}$ in mod-$\bar{F}H$;

\item $\Gamma^{-1}(W) \neq \phi$ for all $W \in \text{ITI}(FG)$; and

\item $\text{ITI}(\bar{F}G) = \bigcup_{W \in \text{ITI}(FG)}\Gamma^{-1}(W)$ is a partition of $\text{ITI}(\bar{F}G)$.

\end{enumerate}
\end{theorem}

\begin{proof}
Here Theorem 2, \cite[Lemma 2.3 and Proposition 1.12]{2} yield (a).
Also Proposition \ref{proposition3} with $J = \bar{F}$ yields (b) and (c) is immediate to conclude the proof.
\end{proof}

\begin{theorem}
\begin{enumerate}[(a)]
\item The map $\Sigma : \text{ITS}(\bar{F}G) \to \text{ITS}(FG)$ such that if $X \in \text{ITS}(\bar{F}G)$, then $\Sigma(X) = W \in \text{ITS}(FG)$ such that there is a $(K,V) \in \mathcal{E}(W)$ such that $V \otimes_K \bar{F} \cong X$ in mod-$\bar{F}G$ is well-defined.
Here $V \in \text{ITS}(KG)$, $Q$ is a vertex of $V$ and $X$.
Set $H = N_G(Q)$ and let $(Q,\mathscr{V})$ (resp. $(Q, \mathscr{Y})$ be a vertex-source pair of $V$ (resp. $X$) then $\mathscr{Y}^h \mid \mathscr{V} \otimes_K \bar{F}$ for some $h \in H$ in mod-$\bar{F}G$, and $\text{Gr}_G^H(X) \mid \text{Gr}_G^H(V) \otimes_K \bar{F}$ in mod-$\bar{F}H$;

\item $\Sigma^{-1}(W) \neq \phi$ for all $W \in \text{ITS}(FG)$; and

\item $\text{ITS}(\bar{F}G) = \bigcup_{W \in \text{ITS}(FG)} \Sigma^{-1}(W)$ is a partition of $\text{ITS}(\bar{F}G)$.
\end{enumerate}
\end{theorem}

\begin{proof}
Let $X \in \text{ITS}(\bar{F}G)$.
Then Theorem \ref{theorem2}, (7) and \cite[Lemma 2.3]{2} imply that $\Sigma$ is a well-defined map.
Also Proposition \ref{proposition3} with $J = \bar{F}$ yields (b) and (c) is immediate to conclude the proof.
\end{proof}

\begin{remark}
As mentioned above, we have completed the proofs of \cite[Theorems 3.1--3.2, 3.4--3.5 and Remarks 3.3 and 3.6]{2}.
In particular, we have obtained a new formula for $(\text{ITS}(\bar{F}G))$ as stated in the abstract.
\end{remark}

\bibliographystyle{abbrv}
\bibliography{equivalence.bib}
	
\end{document}